\documentclass[a4paper, 12pt]{amsart}
\usepackage{amsmath, amssymb, amsthm, microtype}
\usepackage[bottom=1in, left=1in, right=1in, top=1in]{geometry}

\newtheorem{theorem}{Theorem}[section]

\newtheorem{lemma}[theorem]{Lemma}

\newtheorem{prop}[theorem]{Proposition}
\newtheorem{conjecture}[theorem]{Conjecture}

\theoremstyle{remark}

\theoremstyle{definition}

\newcommand{\cA}{\mathcal{A}}
\newcommand{\cE}{\mathcal{E}}

\newcommand{\eps}{\epsilon}
\newcommand{\bN}{\mathbb{N}}

\title[Prime factors of sums of proper divisors]{On the number of prime factors of values of the sum-of-proper-divisors function}

\begin{document}

\author{Lee Troupe}
\address{Department of Mathematics, Boyd Graduate Studies Research Center, University of Georgia, Athens, GA 30602, USA}
\email{ltroupe@math.uga.edu}
\let\thefootnote\relax\footnote{The author was partially supported by NSF Grant DMS-1344994 of the Research and Training Group in Algebra, Algebraic Geometry, and Number Theory, at the University of Georgia.}

\begin{abstract}
Let $\omega(n)$ (resp. $\Omega(n)$) denote the number of prime divisors (resp. with multiplicity) of a natural number $n$. In 1917, Hardy and Ramanujan proved that the normal order of $\omega(n)$ is $\log\log n$, and the same is true of $\Omega(n)$; roughly speaking, a typical natural number $n$ has about $\log\log n$ prime factors. We prove a similar result for $\omega(s(n))$, where $s(n)$ denotes the sum of the proper divisors of $n$: For all $n \leq x$ not belonging to a set of size $o(x)$,
\[
|\omega(s(n)) - \log\log s(n)| < \eps \log\log s(n),
\]
and the same is true for $\Omega(s(n))$.
\end{abstract}

\maketitle

\section{Introduction.}

Let $s(n)$ denote the sum of the proper divisors of a positive integer $n$. The function $s(n)$ has been of interest to number theorists since antiquity; for example, the ancient Greeks wanted to know when $s(n) = n$, calling such integers \emph{perfect}. In modern times, open problems concerning $s(n)$ abound, such as the famous Catalan-Dickson conjecture \cite{dic13}: For any positive integer $n$, the aliquot sequence at $n$ (that is, the sequence of iterates of the function $s$ on $n$) either terminates at 0 or is eventually periodic.

Another conjecture pertaining to $s(n)$ is the following, due to Erd\"os, Granville, Pomerance and Spiro \cite{egps90}:

\begin{conjecture}\label{egps}
If $\cA$ is a set of natural numbers of asymptotic density zero, then $s^{-1}(\cA)$ also has density zero.
\end{conjecture}
Recall now Hardy and Ramanujan's normal order result for $\omega(n)$, where $\omega(n)$ denotes the number of distinct prime divisors of $n$ (see \cite{hr17}):

\begin{theorem}
For any $\eps > 0$ and all numbers $n$ not belonging to a set of asymptotic density zero,
\[
|\omega(n) - \log\log n| < \eps \log\log n.
\]
\end{theorem}

Then Conjecture \ref{egps} would imply the following:

\begin{theorem}\label{normalorder}
For any $\eps > 0$ and all numbers $n \leq x$ not belonging to a set of size $o(x)$,
\[
|\omega(s(n)) - \log\log s(n)| < \eps \log\log s(n).
\]
\end{theorem}

Though Conjecture \ref{egps} remains intractable, we are able to prove Theorem \ref{normalorder} unconditionally in the present article.

If a number $n$ is composite, then $n$ has a proper divisor greater than $n^{1/2}$; so for all composite $n \in (x^{1/2}, x]$, we have $x^{1/4} \leq n^{1/2} \leq s(n) \leq x^2$. Hence, for all but $o(x)$ numbers $n$ up to $x$, $\log\log s(n) = \log\log x + O(1).$
Therefore, it suffices to show that, given $\eps > 0$,
\begin{equation}\label{normalorderx}
|\omega(s(n)) - \log\log x| < \eps \log\log x
\end{equation}
for all $n$ except those belonging to a set of size $o(x)$.

Our normal order result follows from the following estimate.

\begin{theorem}\label{variance}
As $x \to \infty$,
\[
\sum_{n \leq x \,:\, n \notin \cE(x)} (\omega(s(n)) - \log\log x)^2 = o(x(\log\log x)^2).
\]
where $\cE(x) \subset \{1, 2, \ldots, \lfloor x \rfloor\}$ is of size $o(x)$.
\end{theorem}

For large enough $x$, if there are more than $\delta x$ numbers up to $x$ for which (\ref{normalorderx}) fails, where $\delta > 0$ is any positive number, then at least $\tfrac{\delta}{2}x$ of them (say) do not belong to $\cE(x)$, whence
\[
\sum_{n \leq x \,:\, n \notin \cE(x)} (\omega(s(n)) - \log\log x)^2 \geq \tfrac{\delta}{2} \eps^2 x (\log\log x)^2;
\]
this contradicts Theorem \ref{variance} for sufficiently large $x$, and Theorem \ref{normalorder} follows.

\bigskip

\noindent\textbf{Remark.} As one might expect, we obtain a true statement by replacing $\omega(s(n))$ with $\Omega(s(n))$ in the theorem above; this is shown in \S 5.
\bigskip

\noindent\textbf{Notation.} The letters $p$ and $q$ will denote prime variables unless otherwise stated. The symbols $P(n)$ and $P_2(n)$ will denote the largest and second-largest prime factors of $n$, respectively; if $n = 1$, we set $P(n) = P_2(n) = 1$, and if $n$ is prime, we set $P_2(n) = 1$. For $x > 0$, define $\log_1x = \max\{1, \log x \}$, where $\log x$ denotes the natural logarithm of $x$, and let $\log_kx$ denote the $k$th iterate of $\log_1$. (In particular, $\log_k$ is not to be confused with the base-$k$ logarithm.) Sums over $n \notin \cE(x)$ carry the implicit condition that $n \leq x$. The symbols $\ll$ and $\sim$ have their usual meanings, and we will make frequent use of $O$ and $o$-notation. Other notation may be introduced as necessary.

\section{Preliminaries.}

First we show that it suffices to consider a truncated version of the sum in question.

\begin{lemma}\label{truncate}
\[
\sum_{n \leq x} \omega(s(n)) = \left(\sum_{\log_2x < p \leq x^{1/\sqrt{\log_2 x}}} \sum_{\substack{n \leq x \\ p \mid s(n)}} 1\right) + o(x\log_2 x).
\]
\end{lemma}

\begin{proof}
We have
\[
\sum_{n \leq x} \omega(s(n)) = \sum_{n \leq x} \sum_{p : p \mid s(n)} 1.
\]

If $n \leq x$ then $s(n) \leq x^2$. The number of primes $p \notin (\log_2x, x^{1/\sqrt{\log_2x}}]$ dividing a number $m \leq x^2$ is $\leq 2\sqrt{\log_2x} + \pi(\log_2x) = o(\log_2x)$, and the lemma follows.
\end{proof}

\subsection{The exceptional set.} Define $\cE(x) := \{n \leq x : \text{at least one of } A, B, C, D, E, F \text{ fails}\}$, where
\begin{itemize}
 
	\item[A.] $P(n) > x^{1/\log_3x}$,
	\item[B.] $P(n)^2 \nmid n$,
	\item[C.] $P_2(n) > x^{1/\log_3x}$,
	\item[D.] $P_2(n) < xP(n)/2n$,
	\item[E.] $P_2(n)^2 \nmid n$
	\item[F.] if a prime $q$ divides $\gcd(n/P(n), \sigma(n/P(n)))$, then $q < \log_2 x$.
\end{itemize}

\begin{lemma}
$\#\cE(x) = o(x)$.
\end{lemma}

Let $\Psi(x, y)$ denote the count of $y$-smooth numbers up to $x$; that is, the count of numbers $n \leq x$ such that $p \mid n \implies p \leq y$. The following upper bound estimate of de Bruijn \cite[Theorem 2]{db66} will be useful in proving the above lemma.

\begin{prop}\label{smooth}
Let $x \geq y \geq 2$ satisfy $(\log x)^2 \leq y \leq x$. Whenever $u := \frac{\log x}{\log y} \to \infty$, we have
\[
\Psi(x, y) \leq x/u^{u + o(u)}.
\]
\end{prop}

\begin{proof}[Proof of Lemma.]
Let $\cE_j(x) = \{n \leq x : \text{condition } j \text{ fails and all previous conditions hold}\}$, $j \in \{A, B, C, D, E, F\}$. If $n \in \cE_A(x) \cup \cE_B(x)$, then either $P(n) \leq x^{1/\log_3x}$ or $P(n) > x^{1/\log_3x}$ and $P(n)^2 \mid n$. By Proposition \ref{smooth}, the number of $n \leq x$ for which the former holds is $O(x/(\log_2x)^4)$, noting that $(\log_2x)^4 \ll (\log_2x)^{\log_4x} = (\log_3x)^{\log_3x}$. The number of $n \leq x$ for which the latter holds is $\ll x\sum_{p > x^{1/\log_3x}} p^{-2} \ll x\exp(-\log x/\log_3x)$, and this is also $O(x/(\log_2x)^4)$.

If $n \in \cE_C(x)$, we proceed as in \cite[Lemma 2.6]{pol14}. In this case, either (i) $P(n) \leq x^{1/\log_4x}$ or (ii) $P(n) > x^{1/\log_4x}$ and $P_2(n) \leq x^{1/\log_3x}$. The number of $n \leq x$ for which (i) holds is $O(x/(\log_3 x)^{10})$ by Proposition \ref{smooth}. For $n \leq x$ such that (ii) holds, write $n = mP$, where $P = P(n) > x^{1/\log_4x}$ and $P(m) \leq x^{1/\log_3x}$. Then $x/m \geq P > x^{1/\log_4x}$, and the number of such $P$ given $m$ is, by the prime number theorem,
\[
\ll \frac{x/m}{\log(x/m)} \leq \frac{x\log_4x}{m\log x}.
\]
Now we sum this over $m$ such that $P(m) \leq x^{1/\log_3x}$, obtaining that the number of $n \leq x$ for which (ii) holds is
\begin{align*}
\ll \frac{x\log_4x}{\log x} \sum_{m \,:\, P(m) \leq x^{1/\log_3x}} \frac{1}{m} &= \frac{x\log_4x}{\log x} \prod_{p \leq x^{1/\log_3x}} \left(1 - \frac{1}{p}\right)^{-1} \\
	&\ll \frac{x\log_4x}{\log x} \left(\frac{\log_3x}{\log x}\right)^{-1} = \frac{x\log_4x}{\log_3x}.
\end{align*}
Thus $\#\cE_C(x) = O(x\log_4x/\log_3 x)$.

For $\#\cE_D(x)$, note first that $xP(n)/2n > P(n)/2$. Thus any $n$ in $\cE_D(x)$ has a prime factor in $(P(n)/2, P(n))$, and the number of such integers $n$ with $P(n) > x^{1/\log_3(x)}$ is, by the prime number theorem,
\[
\leq x \sum_{q > x^{1/\log_3 x}} \sum_{q/2 < q' < q} \frac{1}{qq'} \ll x \sum_{q > x^{1/\log_3 x}} \frac{1}{q\log q} = O\left(\frac{x\log_3x}{\log x}\right).
\]

To bound $\#\cE_E(x)$, we consider $n \leq x$ with $P_2(n)^2 \mid n$ and $P_2(n) > x^{1/\log_3x}$. The number of such $n$ is $\ll x\sum_{p > x^{1/\log_3x}} p^{-2} \ll x\exp(-\log x/\log_3x)$.

Finally, suppose $q \mid \big(n/P(n), \sigma(n/P(n))\big)$ and $q \geq \log_2x$. Then $q \mid n$ and $q \mid \sigma(n)$. Write $n = q^es$, with $q \nmid s$. If $e \geq 2$, then $n$ has a squarefull divisor $\geq (\log_2x)^2$, and the number of such $n$ is $O(x/\log_2x)$.

So assume $e = 1$. Since $q \mid \sigma(n)$, we have $q \mid \sigma(s)$. Write $s = p_1^{e_1} \cdots p_k^{e_k}$; then $q \mid \sigma(p_i^{e_i})$ for some $i$. If $e_i \geq 2$, then we have $2p_i^{e_i} > \sigma(p_i^{e_i}) \geq q$, so $s$ has a squarefull divisor $\geq q/2$. The number of such $s \leq x/q$ is $O(x/q^{3/2})$, and summing this over $q \geq \log_2x$ gives that the number of possible $n$ is $O(x/\sqrt{\log_2x})$.

If $e_i = 1$, then $p_i \equiv -1 \pmod q$. By Brun-Titchmarsh and partial summation, the number of $s \leq x/q$ divisible by such a prime $p_i$ is
\[
\frac{x}{q} \sum_{\substack{p \leq x/q \\ p \equiv -1 \pmod q}} \frac{1}{p} \ll \frac{x\log_2x}{q^2},
\]
and summing this over $q \geq \log_2x$ we obtain that the number of possible $n$ in this case is $O(x/\log_3x).$

This finishes the proof, noting that all size bounds established are $o(x)$.
\end{proof}

\subsection{Primes in arithmetic progressions.} In the proof of Theorem \ref{sumomegasn}, we will need asymptotic estimates for the count of primes in fairly short progressions. The following theorem is sufficient, though we must exclude moduli which are multiples of a certain integer. As usual, $\pi(X)$ denotes the number of primes $p \leq X$, and $\pi(X; q, a)$ denotes the number of primes $p \leq X$ with $p \equiv a \pmod q$.

%and set $T = x^{1/A}$, where $A > 0$ and $\psi(x) \geq 0$ is a function satisfying $\psi(x) \to 0$ and $x\psi(x) \to \infty$ as $x \to \infty$.

\begin{theorem}\label{siegel}
Let $X$ and $T$ satisfy $X \geq T \geq 2$, and suppose $q \leq T^{2/3}$. Then
\[
\pi(X; q, a) \sim \pi(X)/\varphi(q) \quad \text{ as } \quad \frac{\log X}{\log T} \to \infty
\]
uniformly for all coprime pairs $(a, q)$, except possibly for those $q$ which are multiples of some integer $q_1(T)$. Here $\varphi$ is the usual Euler phi function.
\end{theorem}

\begin{proof}
Referring to the proof of Linnik's theorem in \cite[p. 55]{bom74}, we have
\begin{align*}
\sum_{\substack{p \leq X \\ p \equiv a \pmod q}} \log p = \frac{X}{\varphi(q)} + O\left(\frac{X}{\varphi(q)} \exp(-c_1 A)\right) + O\left(\frac{X\log X}{T}\right) + O\left(\frac{1}{\varphi(q)}X^{1/2}T^5\right),
\end{align*}
unless $q$ is divisible by a certain exceptional modulus $q_1 = q_1(T)$, and where $A$ is such that $T = X^{1/A}$. Notice that the quantity $A$ tends to infinity. Therefore, the first and last $O$-terms are $o(X/\varphi(q))$, as is the middle $O$-term, provided $T > (\log X)^6$; but if $T \leq (\log X)^6$, Theorem \ref{siegel} follows from the Siegel-Walfisz theorem. One now obtains the desired asymptotic for $\pi(X; q, a)$ by standard arguments.
\end{proof}

We will also make use of the following fact \cite[Lemma 2.7]{pol14}:

\begin{prop}\label{pollack27}
Let $q$ be a natural number with $q \leq x^{\tfrac{1}{2\log_3x}}$. The number of $n \leq x$ not belonging to $\cE_A(x) \cup \cE_B(x)$ for which $q$ divides $s(n)$ is
\[
\ll \frac{\tau(q)}{\varphi(q)}\cdot x\log_3x,
\]
where $\tau(q)$ denotes the number of divisors of $q$.
\end{prop}

\section{An average result for $\omega(s(n))$.}

First we prove an average order result for $\omega(s(n))$:

\begin{theorem}\label{sumomegasn}
As $x \to \infty$,
\[
\sum_{n \leq x \,:\, n \notin \cE(x)} \omega(s(n)) \sim x\log\log x.
\]
\end{theorem}

This result will serve as a stepping stone towards Theorem \ref{variance}.

By Lemma \ref{truncate},
\[
\sum_{n \leq x \,:\, n \notin \cE(x)} \omega(s(n)) = \left(\sum_{\substack{\log_2x < p \leq x^{1/\sqrt{\log_2 x}} \\ p \neq q_1(T)}} \sum_{\substack{n \leq x  \,:\, n \notin \cE(x)\\ p \mid s(n)}} 1\right) + o(x\log_2 x),
\]
where $q_1(T)$ is the exceptional modulus coming from Theorem \ref{siegel}. Eventually we will be counting primes in certain progressions modulo $p \leq x^{1/\sqrt{\log_2 x}}$, and so it will suffice to take $T = x^{1.5/\sqrt{\log_2x}}$. To apply Theorem \ref{siegel}, we need that $p$ is not a multiple of $q_1(T)$; but since $p$ is prime, this is the same as requiring $p \neq q_1(T)$, and this excludes at most one summand above. Fix a prime $p \in (\log_2x, x^{1/\sqrt{\log_2x}}]$, $p \neq q_1(T)$ and consider the inner sum above.

Since $n \notin \cE(x)$, we can write $n = mP$, $P := P(n)$, where $P \nmid m$, $P > x^{1/\log_3x}$, $x^{1/\log_3x} < P(m) < x/2m$ (this follows from condition $D$, noting that $m = n/P$), and $p \mid (m, \sigma(m)) \implies p < \log_2x$. Now
\[
\sum_{\substack{n \leq x \,:\, n \notin \cE(x) \\ p | s(n)}} 1 = \sum_{\substack{n \leq x \,:\, n \notin \cE(x) \\ p | s(n) \\ p \mid s(m)}} 1 + \sum_{\substack{n \leq x \,:\, n \notin \cE(x) \\ p | s(n) \\ p \nmid s(m)}} 1.
\]
But this first sum is actually empty! If $p \mid s(n) = s(mP) = Ps(m) + \sigma(m)$, and $p \mid s(m)$, then $p \mid \sigma(m)$ and hence $p \mid m$ (since $m = \sigma(m) - s(m)$). Since $n \notin \cE(x)$, this forces $p < \log_2x$, contradicting our choice of the fixed prime $p$.

If $p \nmid s(m)$, then since $p \mid s(n) = Ps(m) + \sigma(m)$, we have
\[
P \equiv -\sigma(m)s(m)^{-1} \pmod p.
\]

For convenience, we now define the following notation: Write
\[
\sideset{}{'}\sum_m := \sum_{\substack{x^{1/\log_3x} < m \leq x^{1 - 1/\log_3x} \\ p \nmid s(m) \\ x^{1/\log_3x} < P(m) < x/2m \\ P(m)^2 \nmid m \\ q \mid (m, \sigma(m)) \implies q < \log_2x}}.
\]

Now
\[
\sum_{\substack{n \leq x \,:\, n \notin \cE(x) \\ p | s(n) \\ p \nmid s(m)}} 1 = \sideset{}{'}\sum_m \sum_{\substack{P(m) < P \leq x/m \\ P \equiv -\sigma(m)s(m)^{-1} \pmod p}} 1.
\]

The inner sum is equal to
\[
\pi(x/m; p, -\sigma(m)s(m)^{-1}) - \pi(P(m); p, -\sigma(m)s(m)^{-1}).
\]

We now use Theorem \ref{siegel} to rewrite the terms above, with $T = x^{1.5/\sqrt{\log_2x}}$. Note that $P(m), x/m > x^{1/\log_3x}$, so $P(m)$ and $x/m$ are both greater than any fixed power of $T$ for large enough $x$. Since $p \neq q_1(T)$, we have
\begin{align*}
\sideset{}{'}\sum_m &\Big(\pi(x/m; p, -\sigma(m)s(m)^{-1}) - \pi(P(m); p, -\sigma(m)s(m)^{-1})\Big) \\
	&= \frac{1}{p-1}\sideset{}{'}\sum_m \Big(\pi(x/m) - \pi(P(m))\Big) + o\left(\frac{1}{p-1}\sideset{}{'}\sum_m \Big(\pi(x/m) + \pi(P(m))\Big)\right).
\end{align*}

To deal with the $o$-term, observe that from our conditions on $m$,
\begin{align*}
\pi(x/m) + \pi(P(m)) &\leq \pi(x/m) + \pi(x/2m) \\
	&\leq 2\pi(x/m) \leq 10\big(\pi(x/m) - \pi(P(m))\big)
\end{align*}
for large enough $x$, and so the above sum on $m$ is
\begin{align*}
\sim \frac{1}{p-1}\sideset{}{'}\sum_m \Big(\pi(x/m) - \pi(P(m))\Big).
\end{align*}

A prime $P$ counted by the term $\pi(x/m) - \pi(P(m))$ corresponds to an integer $n = Pm$, with $P^2 \nmid n$ and $p \nmid s(m)$; the conditions imposed on $m$ guarantee that $n \leq x$ and $n \notin \cE(x)$. It is clear that every $n \leq x$ with $n \notin \cE(x)$ and $p \nmid s(m)$ will be counted by such a summand. Thus
\begin{align*}
\sideset{}{'}\sum_m \pi(x/m) &- \pi(P(m)) = \#\{n \leq x : n = mP(n) \notin \cE(x), p \nmid s(m)\} \\
	&= \#\{n \leq x : n \notin \cE(x)\} - \#\{n \leq x : n = mP(n) \notin \cE(x), p \mid s(m)\}.
\end{align*}

The following lemma allows us to estimate the subtrahend.

\begin{lemma}\label{pdividessm}
For a fixed prime $\log_2x < q \leq x^{1/\sqrt{\log_2x}}$, the number of $n \leq x$ not belonging to $\cE(x)$ (so $n = mP(n)$) such that $q \mid s(m)$ is 
\[
\ll \frac{x\log_3x\log_4x}{q-1}.
\]
\end{lemma}

\begin{proof}
Write $n = mP$, with $P = P(n)$. Since $n \notin \cE(x)$, $P(m)^2 \nmid m$ and $P(m) > x^{1/\log_3x}$. Note that $q \leq x^{1/\sqrt{\log_2x}}$, and $(x/P)^{\frac{1}{2\log_3x}} > x^{\frac{1}{2(\log_3x)^2}} > x^{1/\sqrt{\log_2x}}$ for large enough $x$. By Proposition \ref{pollack27}, the number of $m \leq x/P$ not in $\cE(x)$ for which $q \mid s(m)$ is
\[
\ll \frac{1}{q-1} \cdot \frac{x\log_3x}{P}.
\]
Summing over $P \in (x^{1/\log_3x}, x]$ gives the result.
\end{proof}

By the above lemma, we have
\[
\sideset{}{'}\sum_m \Big(\pi(x/m) - \pi(P(m))\Big) = x + O\left(\frac{x\log_3x\log_4x}{p-1}\right) + O(\#\cE(x)).
\]

Putting everything together, we have shown that
\[
\sum_{n \leq x \,:\, n \notin \cE(x)} \omega(s(n)) = \sum_{\substack{\log_2x < p \leq x^{1/\sqrt{\log_2 x}} \\ p \neq q_1(T)}} \left( \frac{x}{p-1} + O\left(\frac{x\log_3x\log_4x}{(p-1)^2}\right) + o\left(\frac{x}{p-1}\right)\right).
\]
The $O$-term contributes
\[
\ll x\log_3x\log_4x \sum_{t > \log_2x} \frac{1}{t^2} \ll \frac{x\log_3x\log_4x}{\log_2x} = o(x\log_2x).
\]

By Mertens' first theorem,
\begin{align*}
\sum_{\substack{\log_2x < p \leq x^{1/\sqrt{\log_2 x}} \\ p \neq q_1(T)}} \frac{1}{p-1} &= \log_2(x^{1/\sqrt{\log_2x}}) - \log_2(\log_2(x)) + O(1) \\
	&= (\log_2x)(1 + o(1)).
\end{align*}
Thus,
\[
\sum_{n \leq x \,:\, n \notin \cE(x)} \omega(s(n)) = x\log_2x + o(x\log_2x),
\]
as desired.

\section{Proof of Theorem \ref{variance}.}

Notice that
\begin{align*}
\sum_{n \notin \cE(x)} (\omega(s(n)) - \log_2x)^2 &= \sum_{n \notin \cE(x)} \omega(s(n))^2 - 2\log_2x\sum_{n \notin \cE(x)} \omega(s(n)) + x(\log_2x)^2(1 + o(1)) \\
	&= \sum_{n \notin \cE(x)} \omega(s(n))^2 - x(\log_2x)^2(1 + o(1))
\end{align*}
by Theorem \ref{sumomegasn}. Hence, the following lemma implies Theorem \ref{variance}.

\begin{lemma}\label{sumomegasnsq}
As $x \to \infty$,
\[
\sum_{n \leq x \,:\, n \notin \cE(x)} \omega(s(n))^2 \sim x(\log_2x)^2.
\]
\end{lemma}

We have
\[
\sum_{n \leq x \,:\, n \notin \cE(x)} \omega(s(n))^2 = \sum_{n \notin \cE(x)} \left(\sum_{p \mid s(n)} 1\right)^2 = \sum_{n \notin \cE(x)} \sum_{p, q \,:\, p, q \mid s(n)} 1,
\]
where the sum is over pairs of primes $p, q$. This inner sum is equal to
\[
\sum_{\substack{p, q \,:\, p, q \mid s(n) \\ p \neq q}} 1 + \sum_{p \mid s(n)} 1,
\]
and hence
\[
\sum_{n \leq x \,:\, n \notin \cE(x)} \omega(s(n))^2 = \sum_{n \notin \cE(x)} \sum_{\substack{p, q \,:\, p, q \mid s(n) \\ p \neq q}} 1 + o(x(\log_2x)^2)
\]
by Theorem \ref{sumomegasn}.

We can again truncate the sum in question.

\begin{lemma}
\[
\sum_{n \notin \cE(x)} \omega(s(n))^2 = \left(\sum_{n \notin \cE(x)} \sum_{\substack{\log_2x < p \leq x^{1/\sqrt{\log_2 x}} \\ p \mid s(n)}} \sum_{\substack{\log_2x < q \leq x^{1/\sqrt{\log_2 x}} \\ q \mid s(n) \\ q \neq p}} 1 \right) + o(x(\log_2x)^2).
\]
\end{lemma}

\begin{proof}
By Lemma \ref{truncate}, we have $\omega(s(n)) = \omega'(s(n)) + o(\log_2x)$, where $\omega'(m)$ denotes the number of prime divisors $p$ of $m$ with $p \in (\log_2x, x^{1/\sqrt{\log_2x}}]$. Hence,
\[
\sum_{n \notin \cE(x)} \omega(s(n))^2 = \sum_{n \notin \cE(x)} \omega'(s(n))^2 + o\left(\log_2x\sum_{n \notin \cE(x)} \omega'(s(n))\right) + o(x(\log_2x)^2).
\]
Since $\omega'(s(n)) \leq \omega(s(n))$, we have $\sum_{n \notin \cE(x)} \omega'(s(n)) \ll x\log_2x$ by Theorem \ref{sumomegasn}. The lemma follows.
\end{proof}

\subsection{Proof of Lemma \ref{sumomegasnsq}.}

As before, we fix primes $p, q \in (\log_2x, x^{1/\sqrt{\log_2x}}]$ and count the number of $n \notin \cE(x)$ where $p, q \mid s(n)$. Eventually we will need to use Theorem \ref{siegel} to count primes modulo $pq$, and so we take $T = x^{3/\sqrt{\log_2x}}$ in the theorem. However, we must ensure that $pq$ is not a multiple of some integer $q_1(T)$. Let $p_1(T)$ denote the largest prime factor of $q_1(T)$. If $pq$ is a multiple of $q_1(T)$, then $p_1(T)$ divides $pq$, which forces $p_1(T) = p$ or $p_1(T) = q$. Therefore, we will insist that $p, q \neq p_1(T)$, which excludes at most one each of $p$ and $q$.

%consider $\sum_{\substack{n \notin \cE(x) \\ p, q \mid s(n)}} 1.$

Write $n = mP$, $P := P(n)$, where $P \nmid m$, $P > x^{1/\log_3x}$, $x^{1/\log_3x} < P(m) < x/2m$, and a prime $\ell \mid (n, \sigma(n)) \implies \ell < \log_2x$. As before, we can split the above sum in two, with one sum over $n$ such that $p, q \nmid s(m)$ and the other over $n$ such that either $p$ or $q$ divides $s(m)$; and as before, the latter sum will be empty, by the same argument. Thus we are reduced to considering $n \notin \cE(x)$ with $p, q \mid s(n)$ and $p, q \nmid s(m)$.

If both $p$ and $q$ divide $s(n)$ but not $s(m)$, then since $s(n) = Ps(m) + \sigma(m)$, we have
\[
P \equiv -\sigma(m)s(m)^{-1} \pmod{pq}.
\]
Therefore
\[
\sum_{\substack{n \notin \cE(x) \\ p, q | s(n) \\ p, q \nmid s(m)}} 1 = \sideset{}{'}\sum_m \sum_{\substack{P(m) < P \leq x/m \\ P \equiv -\sigma(m)s(m)^{-1} \pmod{pq}}} 1,
\]
where now $\sideset{}{'}\sum_m$ includes the condition $q \nmid s(m)$.

The inner sum is equal to
\[
\pi\big(x/m; pq, -\sigma(m)s(m)^{-1}\big) - \pi\big(P(m); pq, -\sigma(m)s(m)^{-1}\big).
\]

We once again use Theorem \ref{siegel} on the terms above, with $T = x^{3/\sqrt{\log_2x}}$. The analysis proceeds exactly as before, with the factor of $1/\varphi(p)$ replaced by $1/\varphi(pq)$. We have in the end that the inner sum (over $n$) is asymptotically equal to
\[
\frac{1}{\varphi(pq)}\big(\#\{n \leq x : n \notin \cE(x)\} - \#\{n \leq x : n = mP(n) \notin \cE(x), \text{ either } p \text{ or } q \mid s(m)\}\big).
\]

Applying Lemma \ref{pdividessm} twice, we have that the number of $n \leq x$ not belonging to $\cE(x)$ such that either $p$ or $q$ divides $s(m)$ is
\[
\ll x\log_3x\log_4x\left(\frac{1}{p-1} + \frac{1}{q-1}\right).
\]

Putting everything together, we have shown that
\begin{align*}
\sum_{n \leq x \,:\, n \notin \cE(x)} &\omega(s(n))^2 = x\sum_{p} \sum_{q} \frac{1}{(p-1)(q-1)} \\ & + \sum_{p}\sum_{q} O\left(\frac{x\log_3x\log_4x}{(p-1)^2(q-1)} + \frac{x\log_3x\log_4x}{(p-1)(q-1)^2}\right) + o(x(\log_2 x)^2),
\end{align*}
where the sums over $p$ and $q$ have the restrictions $p, q \in (\log_2x, x^{1/\sqrt{\log_2x}}], q \neq p$, and $p, q \neq p_1(T)$ (from Theorem \ref{siegel}).
Now, since $q \leq x^{1/\sqrt{\log_2x}}$ and $\sum_p 1/p^2$ converges,
\[
\sum_p \sum_q \frac{1}{(p-1)^2(q-1)} = O(\log_2x)
\]
by Mertens' first theorem. Hence the $O$-term contributes
\[
\ll x\log_2x\log_3x\log_4x = o(x(\log_2x)^2),
\]
and so
\[
\sum_{n \leq x \,:\, n \notin \cE(x)} \omega(s(n))^2 = x\sum_{p} \sum_{q} \frac{1}{(p-1)(q-1)} + o(x(\log_2x)^2).
\]
Another application of Mertens' theorem tells us
\[
\sum_{n \leq x \,:\, n \notin \cE(x)} \omega(s(n))^2 = x(\log_2x)^2 + o(x(\log_2x)^2),
\]
which completes the proof.

\section{From $\omega(s(n))$ to $\Omega(s(n))$}

We conclude by showing that the result proved in the previous section holds with $\omega(s(n))$ replaced by $\Omega(s(n))$.

\begin{lemma}\label{Omega2}
\[
\sum_{n \notin \cE(x)} \big(\Omega(s(n)) - \omega(s(n))\big)^2 = o(x(\log_2x)^2).
\]
\end{lemma}

It follows quickly from this lemma that $\sum_{n \notin \cE(x)} \big(\Omega(s(n)) - \log_2x\big)^2 = o(x(\log_2x)^2)$. Adding and subtracting $\omega(s(n))$ inside the square and expanding, we have
\begin{align*}
\sum_{n \notin \cE(x)} &\big(\Omega(s(n)) - \log_2x\big)^2 = \sum_{n \notin \cE(x)} \big(\Omega(s(n)) - \omega(s(n))\big)^2 + \sum_{n \notin \cE(x)} \big(\omega(s(n)) - \log_2x\big)^2 \\ &+ 2\sum_{n \notin \cE(x)} \big[\big(\Omega(s(n)) - \omega(s(n))\big)\big(\omega(s(n)) - \log_2x\big)\big].
\end{align*}

The first and second sums are $o(x(\log_2x)^2)$, by the above lemma and Theorem \ref{variance}, respectively. An application of the Cauchy-Schwarz inequality shows us that the last term squared is 
\[
\ll \left(\sum_{n \notin \cE(x)} \big(\Omega(s(n)) - \omega(s(n))\big)^2\right)\left(\sum_{n \notin \cE(x)} \big(\omega(s(n)) - \log_2x\big)^2\right);
\]
using Lemma \ref{Omega2} and Theorem \ref{variance} once more and taking square roots, we see that this is also $o(x(\log_2x)^2)$.

\begin{proof}[Proof of Lemma \ref{Omega2}]

We wish to estimate from above the quantity
\begin{align*}
\sum_{n \notin \cE(x)} \big(\Omega(s(n)) - \omega(s(n))\big)^2 &= \sum_{n \leq x \,:\, n \notin \cE(x)} \left(\sum_{\substack{p^k \mid s(n) \\ k \geq 2}} 1 \right)^2 \\
	&= \sum_{n \notin \cE(x)} \sum_{\substack{p^k, q^j \mid s(n) \\ k, j \geq 2 \\ p \neq q}} 1 + \sum_{n \notin \cE(x)} \sum_{\substack{p^k, p^j \mid s(n) \\ k, j \geq 2}} 1.
\end{align*}
We handle the ``$p = q$'' sum first. If $p^k \mid s(n)$ and $j \leq k$, the condition $p^k, p^j \mid s(n)$ is satisfied $k-1$ times given $p$, and so the sum is
\[
\ll 2 \sum_{n \notin \cE(x)} \sum_{\substack{p^k \mid s(n) \\ k \geq 2}} k = \sum_{k \geq 2} k \sum_{n \notin \cE(x)} \sum_{p \;:\; p^k \mid s(n)} 1.
\]
A number $m \leq x^2$ has at most $\tfrac{20}{k}\log_3x$ primes $p > x^{1/10\log_3x}$ with $p^k \mid m$. Define $L_1 := \lfloor 30\log_3x \rfloor$. If $p > x^{1/10\log_3x}$, then $p^{L_1} > x^2$, so certainly for $k > L_1$ the condition $p^k \mid s(n)$ cannot be met. Therefore,
\begin{align*}
\sum_{k \geq 2} k \sum_{n \notin \cE(x)} \sum_{\substack{p > x^{1/10\log_3x} \\ p^k \mid s(n)}} 1 \leq 20x\log_3x \sum_{k=2}^{L_1} 1 = O(x\log_3x \cdot L_1),
\end{align*}
and $x\log_3x \cdot L_1 \ll x(\log_3x)^2$.

It remains to consider
\[
\sum_{n \notin \cE(x)} \sum_{\substack{p \leq x^{1/10\log_3x} \\ p^k \mid s(n) \\ k \geq 2}} k = \sum_{p \leq x^{1/10\log_3x}} \sum_{\substack{n \notin \cE(x) \\ p^k \mid s(n) \\ k \geq 2}} k = \sum_{p \leq x^{1/10\log_3x}} \sum_{\substack{n \notin \cE(x) \\ p^k \mid s(n) \\ p^k \leq x^{1/2\log_3x} \\ k \geq 2}} k + \sum_{p \leq x^{1/10\log_3x}} \sum_{\substack{n \notin \cE(x) \\ p^k \mid s(n) \\ p^k > x^{1/2\log_3x} \\ k \geq 2}} k.
\]
For the first sum, the condition $p^k \leq x^{1/2\log_3x}$ allows us to apply Proposition \ref{pollack27}, which gives
\[
\sum_{p \leq x^{1/10\log_3x}} \sum_{\substack{n \notin \cE(x) \\ p^k \mid s(n) \\ p^k \leq x^{1/2\log_3x} \\ k \geq 2}} k \ll x\log_3x \sum_{k = 2}^{L_2} \sum_{p \leq x^{1/10\log_3x}} \frac{k^2}{p^k},
\]
with $L_2 = \lfloor 3 \log x \rfloor$. But the sum of $k^2/p^k$ over all $p \geq 2$ and all $k \geq 2$ is $O(1)$, so this is $O(x\log_3x)$.

For the second sum, define $\ell(p) := \max\{m \in \bN : p^m \leq x^{1/2\log_3x}\}$. Notice that $\ell(p) < \log x$ trivially and that $p^{\ell(p)} > x^{1/2\log_3x}/p$. The second sum is bounded from above by
\[
\sum_{\substack{p \leq x^{1/10\log_3x} \\ p^k > x^{1/2\log_3x} \\ k \geq 2}} \sum_{\substack{n \notin \cE(x) \\ p^{\ell(p)} \mid s(n)}} k \ll x\log_3x \sum_{k = 2}^{L_2} k \sum_{p \leq x^{1/10\log_3x}} \frac{\ell(p)}{p^{\ell(p)}},
\]
using Proposition \ref{pollack27} once more. The above inequalities then show that this sum is
\[
\ll \frac{x(\log x)^3\log_3x}{x^{1/2\log_3x}} \cdot x^{1/10\log_3x} = o(x).
\]
Thus $\sum_{n \notin \cE(x)}\big(\Omega(s(n)) - \omega(s(n))\big) \ll x\log_3x\log_4x$, which is $o(x\log_2x)$.

We now turn our attention to $\sum_{n \notin \cE(x)} \sum_{p^k, q^j \mid s(n)} 1$, with $k, j \geq 2$ and $p \neq q$. Arguments similar to those used before show
\[
\sum_{n \notin \cE(x)} \sum_{\substack{p^k \mid s(n) \\ k \geq 2}} \sum_{\substack{q^j \mid s(n) \\ j \geq 2 \\ q \neq p}} 1 = \sum_{n \notin \cE(x)} \sum_{\substack{p^k \mid s(n) \\ k \geq 2}} \sum_{\substack{q \leq x^{1/10\log_3x} \\ q^j \mid s(n) \\ j \geq 2 \\ q \neq p}} 1 + o(x\log_2x\log_3x),
\]
and $p$ can certainly be restricted in the same way. Rearranging the sum, we have
\[
\sum_{n \notin \cE(x)} \sum_{\substack{p \leq x^{1/10\log_3x} \\ p^k \mid s(n) \\ k \geq 2}} \sum_{\substack{q \leq x^{1/10\log_3x} \\ q^j \mid s(n) \\ j \geq 2 \\ q \neq p}} 1 = \sum_{p \leq x^{1/10\log_3x}} \sum_{\substack{q \leq x^{1/10\log_3x} \\ q \neq p}} \sum_{\substack{n \notin \cE(x) \\ p^k, q^j \mid s(n) \\ k, j \geq 2}} 1.
\]
We now proceed in essentially the same way as before. Write the inner sum over $q$ as two sums, one over $q^j \leq x^{1/4\log_3x}$ and the other over $q^j > x^{1/4\log_3x}$. Do the same for the outer sum over $p$. When the dust settles, we are left with four sums to handle, one for each possible combination of ranges for $p^k$ and $q^j$.

First we handle the case $p^k, q^j \leq x^{1/4\log_3x}$. Using Proposition \ref{pollack27}, we obtain
\begin{align*}
\sum_{p \leq x^{1/10\log_3x}} \sum_{\substack{q \leq x^{1/10\log_3x} \\ q \neq p}} \sum_{\substack{n \notin \cE(x) \\ p^k, q^j \mid s(n) \\ p^k, q^j \leq x^{1/4\log_3x} \\ k, j \geq 2}} 1 &\ll x\log_3x\sum_{k \geq 2} \sum_{j \geq 2} \sum_{p \leq x^{1/10\log_3x}} \sum_{\substack{q \leq x^{1/10\log_3x} \\ q \neq p}} \frac{kj}{p^kq^j} \\
 &\ll x\log_3x.
\end{align*}

If $p^k, q^j > x^{1/4\log_3x}$, we define $\ell(p) = \max\{m : p^m \leq x^{1/4\log_3x}\}$ and, as before, obtain that this sum is
\[
\ll x\log_3x \sum_{k \geq 2} \sum_{j \geq 2} \sum_{p \leq x^{1/10\log_3x}} \sum_{q \leq x^{1/10\log_3x}} \frac{\ell(p)\ell(q)}{p^{\ell(p)}q^{\ell(q)}} \ll \frac{x(\log x)^4\log_3x}{x^{1/2\log_3x}}\cdot x^{1/5\log_3x},
\]
which is $o(x)$.

Assume now that $q^j \leq x^{1/4\log_3x} \leq p^k$; the final case is completely similar. Then this sum is bounded from above by
\begin{align*}
\sum_{\substack{p \leq x^{1/10\log_3x} \\ p^k > x^{1/4\log_3x} \\ k \geq 2}} \sum_{\substack{q \leq x^{1/10\log_3x} \\ q^j \leq x^{1/4\log_3x} \\ j \geq 2 }} \sum_{\substack{n \notin \cE(x) \\ p^{\ell(p)}, q^j \mid s(n)}} 1 &\ll x\log_3x \sum_{k = 2}^{L_2} \sum_{j \geq 2} \sum_{p \leq x^{1/10\log_3x}} \sum_{\substack{q \leq x^{1/10\log_3x} \\ q \neq p}} \frac{\ell(p)j}{p^{\ell(p)}q^j} \\
	&\ll \frac{x(\log x)^2 \log_3x}{x^{1/3\log_3x}}. \qedhere
\end{align*}
\end{proof}

\noindent\textbf{Remark.} We know, by the celebrated Erd\H os - Kac theorem, that (roughly speaking) $\omega(n)$ is normally distributed with mean and variance $\log\log n$. The corresponding theorem for $\omega(\sigma(n))$ follows from methods of Erd\H os and Pomerance \cite{ep85}: $\omega(\sigma(n))$ is also normally distributed, but with mean $\tfrac{1}{2}(\log\log n)^2$ and standard deviation $\tfrac{1}{\sqrt{3}}(\log\log{n})^{3/2}$. One hopes that something similar can be said about $\omega(s(n))$. The results of the present article indicate that $s(n)$ typically has just as many prime factors as $n$; for this reason, we expect the Erd\H os - Kac theorem to hold with $\omega(s(n))$ in place of $\omega(n)$.

\section*{Acknowledgements}

The problem tackled here first arose in conversations between Paul Pollack and Carl Pomerance; the author is indebted to both. The author would like to extend further gratitude to Paul Pollack for advice and suggestions in the course of writing this article.

\bibliographystyle{amsalpha}
\bibliography{refs}

\end{document}